\documentclass[a4paper,12pt]{article}
\usepackage[utf8]{inputenc}
\usepackage[cm]{aeguill}
\usepackage{amsfonts}
\usepackage{latexsym}
\usepackage[dvips]{graphicx}

\usepackage{amsmath}
\usepackage{amssymb}
\usepackage{amsthm}
\usepackage{bbm}
\usepackage{nicefrac}
\usepackage[T1]{fontenc}
 
\newtheorem{df}{Definition}[section]
\newtheorem{lm}[df]{Lemma}
\newtheorem{pr}[df]{Proposition}
\newtheorem{Th}[df]{Theorem}
\newtheorem{co}[df]{Corollary}
\newtheorem{rem}[df]{Remark}

\newcommand{\e}{\varepsilon}

\newcommand{\R}[1]{\mathbb{R}^{#1}}

\newcommand{\s}{\mathcal{S}}
\newcommand{\M}{\mathcal{M}}
\newcommand{\MM}{\widehat{\mathcal{M}}}

\renewcommand{\L}{\mathcal{L}}

\renewcommand{\AA}{\widetilde{\mathcal{A}}}
\newcommand{\AAA}{\widehat{\mathcal{A}}}
\newcommand{\HH}{\mathcal{H}}
\newcommand{\HHH}{\widehat{\mathcal{H}}}

\newcommand{\om}{\omega}

\newcommand{\f}{\varphi}
\newcommand{\F}{\Phi}

\renewcommand{\d}{\operatorname{d}}
\newcommand{\supp}{\operatorname{supp}}

\author{M. ZAVIDOVIQUE \\ \small \texttt {UMPA, ENS Lyon, 46 all\' ee d'Italie, 69007, Lyon, France} \\ \small \texttt{e-mail: maxime.zavidovique@umpa.ens-lyon.fr}}
\title{Weak KAM for commuting Hamiltonians}

\begin{document}

\maketitle

\begin{abstract}
For two commuting Tonelli Hamiltonians, we recover the commutation of the Lax-Oleinik semi-groups, a result of Barles and Tourin (\cite{barles}), using a direct geometrical method (Stoke's theorem). We also obtain a "generalization" of a theorem of Maderna (\cite{Mad}). More precisely, we prove that if the phase space is the cotangent of a compact manifold then the weak KAM solutions (or viscosity solutions of the critical stationary Hamilton-Jacobi equation) for $G$ and for $H$ are the same. As a  corollary we obtain the equality of the Aubry sets and of the Peierls barrier. This is also related to works of Sorrentino (\cite{So}) and Bernard (\cite{BeSy}).
\end{abstract}

\section*{Introduction}
It has been known for quite some time that the existence of first integrals affects the dynamics of Hamiltonian flows on the cotangent of a manifold. Indeed, the famous Arnol'd-Liouville theorem (\cite{Ar}) states the remarkable fact that under very mild compactness and connectedness conditions, if a Tonelli Hamiltonian $H$ defined on the cotangent of an $n$-dimensional manifold $M$ has $n$ everywhere independent first integrals in involution, then the manifold is necessarily a torus. Moreover the Hamiltonian flow is conjugated to a geodesic flow and $T^*M$ is foliated by invariant tori on which the flow is linear.\\
In the past decades, new techniques have been developed in order to study the dynamics of a single Tonelli Hamiltonian and existence of invariant sets. Aubry-Mather theory (see \cite{Ma},\cite{MaZ}, \cite{Man},\cite{Ba} for introductions) has had a huge development. More recently, thanks to Albert Fathi's weak KAM theory (see \cite{Fa}, \cite{FaMa} for introductions and \cite{FaSi}, \cite{Be1}, \cite{Mad} for further developments) the link between the geometrical point of view of Aubry-Mather theory and the widely studied PDE approach of Hamilton-Jacobi equations has  allowed to simplify the proofs of already known results (in both fields) and obtain new ones (see for example \cite{Fa1},\cite{Fa2},\cite{fatfigrif},\cite{Be2}). Moreover, a discrete version of weak KAM has already appeared fruitful in the related subject of optimal transportation (\cite{BeBu},\cite{BeBu1}, \cite{BeBu2} and \cite{fatfig07}).
\\
The connection between Aubry-Mather theory and first integrals has not, to our knowledge, yet been much studied. First results (although not formulated this way) appear in \cite{BeSy} where it is shown that given a Tonelli Hamiltonian $H$ on the cotangent space of a closed compact Manifold, the Aubry, Mather and Ma\~n\' e sets are symplectic invariants. This may be directly applied to the Hamiltonian flows of Tonelli first integrals of $H$ which are exact symplectomorphisms which preserve $H$. Recently, in \cite{So}, it is shown thanks to Aubry-Mather theory that in the Arnol'd-Liouville theorem, if the involution hypothesis between the first integrals is dropped, much information can still be recovered on the dynamics of $H$ and on its first integrals.\\
From the PDE point of view, in \cite{barles}, the authors study on $M=\R{n}$ the so called multi-time Hamilton Jacobi equation, that is, given Tonelli Hamiltonians $H_1, \ldots, H_k$ and an initial value $u_0:\R{n}\to \R{}$, they look for solutions $u:\R{n}\times \R{k}\to \R{}$ of the equation
\begin{eqnarray*}
\forall x\in \R{n},\  u(x,0,\ldots ,0)&=&u_0(x),\\
\frac{\partial u}{\partial t_1} +H_1(x,\d_x u)&=&0,\\
& \vdots &  \\
\frac{\partial u}{\partial t_k} +H_k(x,\d_x u)&=&0.\\
\end{eqnarray*}
By proving existence of such functions, they actually obtain a commutation property for the Lax-Oleinik semi-groups used in weak KAM theory. The same problem is studied under less stringent regularity hypothesis in \cite{MoRa}.
Let us now explain the setting we use and the results we obtain. Let $M$ be a finite dimensional $C^2$ complete connected Riemmanian manifold. We will say that a Hamiltonian $H:T^*M\to \R{}$ is Tonelli if it is $C^2$ and if it verifies the following conditions:
\begin{enumerate}
\item\textbf{uniform superlinearity:} for every $K>0$, there exists $C^*(K)\in \R{}$ such that 
$$\forall (x,p)\in T^*M,\  H(x,p)\geqslant K\|p\|-C^*(K),$$
\item\textbf{uniform boundedness}: for every $R\geqslant 0$, we have
$$A^*(R)=\sup \{H(x,p),\|p\|\leqslant R\}<+\infty,$$
\item\textbf{$C^2$-strict convexity in the fibers:} for every $(x,p)\in T^*M$, the second derivative along the fibers $\partial ^2H/\partial p^2 (x,p)$ is positive strictly definite.
\end{enumerate} 
We recall that $T^*M$ is equipped with a canonical symplectic structure by setting $\Omega=-\d \lambda$ where $\lambda$ is the canonical Liouville form. We may then define the Hamiltonian vector-field $X_H$ by 
$$\forall (x,p)\in T^*M,\  \Omega(X_H(x,p), .)=\d_{(x,p)}H.$$ 
We then may define the Lax-Oleinik semi-group: if $u:M\to \R{}\cup \{-\infty,+\infty\}$ is a function, we set
$$\forall x \in M,\ \forall s>0,\  T_H^{-s}u (x)=\inf_{\substack{\gamma \\ \gamma(s)=x}}u(\gamma(0))+\int_0^s L_H(\gamma(\sigma),\dot{\gamma}(\sigma))\d\sigma,$$
where the infimum is taken over all absolutely continuous curves reaching $x$ and where $L_H:TM\to \R{}$ is the Lagrangian associated with $H$  defined by
$$\forall (x,v)\in TM,\  L_H(x,v)=\max_{p\in T^*M} p(v)-H(x,p).$$
Now, if $G$ and $H$ are two Tonelli Hamiltonians, we will say that $G$ and $H$ Poisson commute if the function $\Omega (X_G,X_H)$ vanishes everywhere. We will give a direct proof of the following theorem which also results from \cite{barles} for $M=\R{n}$ and from \cite{viterbo}.
\begin{Th}\label{comu}
If $G$ and $H$ are two Tonelli Hamiltonians which Poisson commute, then their Lax-Oleinik semi-groups commute.
\end{Th}
In order to state the second theorem, we need to introduce the notion of weak KAM solution:
\begin{df}\rm
We say that a function $u:M\to\R{}$ is a weak KAM solution for $H$ (resp. $G$) if and only if there is a constant $\alpha\in \R{}$ such that for any $t>0$, we have
$$u=T_H^{-t}u +t\alpha,$$
(resp. $u=T_G^{-t}u +t\alpha$).
\end{df}
\begin{Th}\label{inva}
If $M$ is compact, then any weak KAM solution for $G$ is a weak KAM solution for $H$.
\end{Th}
We then introduce the notion of subsolution with the following definition:
$u:M\to\mathbb{R}$ is an $(\alpha,H)$-subsolution (resp. $(\alpha,G)$-subsolution) 
 if
$$\forall (x,y,t)\in M^2\times \mathbb{R}_+,\  u(y)-u(x)\leqslant \inf \int_0^t L_H(\gamma(\sigma),\dot{\gamma}(\sigma))\d\sigma
+t\alpha.$$
 where the infimum is taken on all absolutely continuous curves $\gamma :[0,t]\to M$ such that $\gamma(0)=x$ and $\gamma(t)=y$
(resp. $ u(y)-u(x)\leqslant \inf\int_0^t L_G(\gamma(\sigma),\dot{\gamma}(\sigma))\d\sigma
+t\alpha.$).
\begin{Th}\label{subs}
If there exists a  function $u:M\to\R{}$ which is both an $(\alpha,G)$-subsolution and an $(\alpha',H)$-subsolution for some constants $(\alpha,\alpha')$ then there is a $C^{1,1}$ function $u'$ which is also both an $(\alpha,G)$-subsolution and an $(\alpha',H)$-subsolution.
\end{Th}

In \cite{Fa1}, Fathi gives a canonical way to pair positive weak KAM solutions with negative weak KAM solutions in the compact case. We prove in the last section that this pairing is the same for commuting Hamiltonians (see theorem \ref{same}). As a  corollary, we establish that the Aubry sets, the Ma\~n\' e sets and the Peierls barrier (defined in section \ref{4}) coincide for both Hamiltonians.

Finally, in the last section, we study some links between the Mather $\alpha$ functions (or effective Hamiltonians) of commuting Hamiltonians. More precisely, we show that their flat parts are the same.

While this paper was being written, similar results were obtained independently by X. Cui and J. Li.  in a preprint. 
For the current version of their work see \cite{chinois2}.

\section*{Acknowledgment}
I first would like to thank Nalini Anantharaman for pointing out to me that it would be interesting to study weak KAM theory for commuting Hamiltonians.
I would like to thank Albert Fathi for his careful reading of the manuscript and for his comments and remarks during my research on this subject. I also thank Bruno S\' evennec for useful discussions. This paper was partially elaborated during a stay at the Sapienza University in Rome. I wish to thank Antonio Siconolfi, Andrea Davini for useful discussions  and the Dipartimento di Matematica "Guido Castelnuovo" for its hospitality while I was there. I also would like to thank Explora'doc which partially supported me during this stay. Finally, I would like to thank the ANR KAM faible (Project BLANC07-3\_187245, Hamilton-Jacobi and Weak KAM Theory) for its support during my research.

\section{Commutation property for the Lax-Oleinik semi-groups}
Let $M$ be a $C^2$ complete connected manifold. In the sequel, we will denote by $\lambda$ the canonical Liouville form defined on $TT^*M$ and by $\Omega=-\d \lambda$ the canonical symplectic form. Let $H$ and $G$ be two Tonelli Hamiltonians which commute. More precisely, this means that $\{G,H\}=\Omega(X_G,X_H)=0$ where $\{.,.\}$ denotes the Poisson bracket and $X_G$, $X_H$ the Hamiltonian vector-fields of $G$, $H$. By basic properties of the Poisson bracket, we have $[X_G,X_H]=X_{\{G,H\}}=0$. Therefore the Hamiltonian flows commute. Finally, from $\Omega(X_G,X_H)=0$, by definition of the Hamiltonian vector-field, we deduce that $\d H(X_G)=0$ and $\d G(X_H)=0$ which means that $G$ is constant on the trajectories of $X_H$ (or in other terms, $G$ is a first integral of $H$) and vice versa.

We will denote by $L_G$ and $L_H$ the Lagrangians associated with $G$ and $H$ and by $\L_G$ and $\L_H$ the respective Legendre transforms, $\f_G$, $\F_G$ and $\f_H$, $\F_H$ will be respectively the Lagrangian and Hamiltonian flows of respectively $G$ and $H$. Finally, $T_G^-$ and $T_H^-$ are the Lax-Oleinik semi-groups associated with $L_G$ and $L_H$, that is if $u:M\to \R{}\cup \{-\infty,+\infty\}$ is a function,
$$\forall x \in M,\forall s>0,\  T_G^{-s}u (x)=\inf u(\gamma (0))+\int_0^s L_G(\gamma(\sigma),\dot{\gamma}(\sigma))\d\sigma,$$
where the infimum is taken over all absolutely continuous curves $\gamma : [0,s] \to M$ with $\gamma(s)=x$. Obviously, the definition of the Lax-Oleinik semi-group associated with $H$ is similar. For an exposition of these definitions, see \cite{Fa}.
\\
The following has already appeared in \cite{barles} in a different setting, with a different formulation and in \cite{MoRa} for less regular Hamiltonians  (see also \cite{viterbo} for related results). It is mainly \ref{comu}, let us reformulate it:
\begin{Th}[G. Barles-A. Tourin]\label{com}
 The Lax-Oleinik semi-groups commute, that is, if $u:M\to \R{}\cup \{-\infty,+\infty\}$ is a function and $s,t$ are two positive real numbers then 
$$T_G^{-s}T_H^{-t}u=T_H^{-t}T_G^{-s}u.$$
\end{Th}
 In order to prove this statement, let us introduce the action functionals (we define it here for $G$, the definition for $H$ is the same):
 \begin{df}\rm
 Let $s>0$ and $(x,y)\in M^2$, then we set
 $$A_G^{s} (x,y)=\inf \int_0^s L_G(\gamma(\sigma),\dot{\gamma}(\sigma))\d\sigma,$$
 where the infimum is taken on all absolutely continuous curves $\gamma : [0,s]\to M$ with $\gamma(0)=x$ and $\gamma(s)=y$.
 \end{df}
The proof of \ref{com} will be  a straight consequence of the following lemma:

\begin{lm}\label{com1}
Let $s,t>0$ be two positive real numbers then the following holds:
$$\forall(x,z)\in M^2,\  \inf_{y\in M} A_G^s(x,y)+A_H^t(y,z)=\inf_{y\in M} A_H^t(x,y)+A_G^s(y,z).$$
\end{lm}
\begin{proof}
Let us begin by recalling that the action functionals are locally semi-concave functions (see \cite{Fa} or \cite{Be2}) and therefore, if $(x_0,z_0)\in M^2$ and if $y_0$ reaches the infimum (which is always the case for some $y_0$) in the following:
$$\inf_{y\in M} A_G^s(x_0,y)+A_H^t(y,z_0)=A_G^s(x_0,y_0)+A_H^t(y_0,z_0),$$
then the following is verified
\begin{equation}\label{EL}
\frac{\partial A_G^s}{\partial y}(x_0,y_0)+\frac{\partial A_H^t}{\partial x}(y_0,z_0)=0
\end{equation}
and the partial derivatives do exist. Actually, more can be said. Let $\gamma_1$ and $\gamma_2$ verify that $\gamma_1(0)=x_0$, $\gamma_1(s)=y_0$, $\gamma_2(0)=y_0$, $\gamma_2(t)=z_0$ and
$$A_G^s(x_0,y_0)+A_H^t(y_0,z_0)=\int_0^s L_G(\gamma_1(\sigma),\dot{\gamma}_1(\sigma))\d\sigma+\int_0^t L_H(\gamma_2(\sigma),\dot{\gamma}_2(\sigma))\d\sigma.$$
The following holds (see \cite{Fa} proposition 4.11.1, \cite{fatfig07}  corollary B.20 or \cite{Be2}):
$$\left( y_0,\frac{\partial A_G^s}{\partial y}(x_0,y_0)\right)=\L_G (y_0,\dot{\gamma}_1(s))$$ 
and
$$\left( y_0 ,-\frac{\partial A_H^t}{\partial x}(y_0,z_0)\right)=\L_H (y_0,\dot{\gamma}_2(0)).$$
Finally, using the fact that the $\gamma_i$ are  minimizers hence trajectories of the respective Euler-Lagrange flows and (\ref{EL}), setting
$$p_0=\frac{\partial A_G^s}{\partial y}(x_0,y_0)=-\frac{\partial A_H^t}{\partial x}(y_0,z_0)$$
 we obtain that:
\begin{equation}
\forall \sigma\in [0,s],\  \L_G (\gamma_1(\sigma),\dot{\gamma}_1(\sigma))=\L_G(\f_G^{\sigma -s}(y_0,\dot{\gamma}_1(s)))=\F_G^{\sigma-s}(y_0,p_0),
\label{sup}
\end{equation}
\begin{equation}
\forall \sigma\in [0,t],\  \L_H (\gamma_2(\sigma),\dot{\gamma}_2(\sigma))=\L_H(\f_H^{\sigma}(y_0,\dot{\gamma}_2(0)))=\F_H^{\sigma}(y_0,p_0).
\label{sup2}
\end{equation}
Using the definition of $L_G$ we have
\begin{multline*}
A_G^s(x_0,y_0)
=\int_0^s L_G (\gamma_1(\sigma),\dot{\gamma}_1(\sigma)))\d\sigma\\
=\int_0^s \frac{\partial L_G}{\partial v}(\gamma_1(\sigma),\dot{\gamma}_1(\sigma))\dot{\gamma}_1(\sigma)\d\sigma-\int_0^s G\left(\gamma_1(\sigma), \frac{\partial L_G}{\partial v}(\gamma_1(\sigma),\dot{\gamma}_1(\sigma))\right)\d\sigma.
\end{multline*}
We now recognize in the first integral the image of the Hamiltonian vector-field under the Liouville form. Therefore, also using \ref{sup}, that Lagrangian and Hamiltonian flows are conjugated by the Legendre transform and that $G$ is constant on its Hamiltonian trajectories, we get that
\begin{multline*}
A_G^s(x_0,y_0)\\
=\int_0^s \frac{\partial L_G}{\partial v}(\gamma_1(\sigma),\dot{\gamma}_1(\sigma))\dot{\gamma}_1(\sigma)\d\sigma-\int_0^s G\left(\gamma_1(\sigma), \frac{\partial L_G}{\partial v}(\gamma_1(\sigma),\dot{\gamma}_1(\sigma))\right)\d\sigma
\\
=\int_0^s \lambda (X_G (\F_G^{\sigma-s}(y_0,p_0)))d\sigma-\int_0^s G\left(\F_G^{\sigma-s}(y_0,p_0)\right)\d\sigma\\
=\int_0^s \lambda (X_G (\F_G^{\sigma-s}(y_0,p_0)))\d\sigma-sG(y_0,p_0).
\end{multline*}
Reasoning along the same lines yields similarly that
\begin{equation*}
A_H^t(y_0,z_0)=\int_0^t \lambda (X_H (\F_H^{\sigma}(y_0,p_0)))\d\sigma-tH(y_0,p_0).
\end{equation*}
Summing up, we have proved that
\begin{multline}\label{sum}
\inf_{y\in M} A_G^s(x_0,y)+A_H^t(y,z_0)\\
=\int_0^s \lambda (X_G (\F_G^{\sigma-s}(y_0,p_0)))\d\sigma+\int_0^t \lambda (X_H (\F_H^{\sigma}(y_0,p_0)))\d\sigma\\
-sG(y_0,p_0)-tH(y_0,p_0).
\end{multline}
Now, let us set $(y_1,p_1)=\F_H^t\circ \F_G^{-s}(y_0,p_0)$. we define 
$$\forall \sigma \in[0,s],\  \gamma_3(\sigma)=\pi _1(\F_G^{\sigma}(y_1,p_1))$$
and
$$\forall \sigma \in[0,t], \  \gamma_4(\sigma)=\pi_1(\F_H^{\sigma-t}(y_1,p_1)),$$
where $\pi_1:T^*M\to M$ denotes the canonical projection on the manifold. 
First of all, let us notice that since $G$ and $H$ Poisson commute, then their Hamiltonian vector-fields also commute, which means that the Hamiltonian flows commute. As a direct consequence, we have that $\gamma_3(s)=z_0$. Moreover, it is obvious from the definitions that $\gamma_4(0)=x_0$.
Let us now compute the quantity $A$ defined below. The same arguments as those exposed previously give that
\begin{multline*}
A=\int_0^tL_H\left(\gamma_4(\sigma),\dot{\gamma}_4(\sigma)\right)\d\sigma+\int_0^sL_G\left(\gamma_3(\sigma),\dot{\gamma}_3(\sigma)\right)\d\sigma\\
=\int_0^t \lambda (X_H (\F_H^{\sigma-t}(y_1,p_1)))\d\sigma+\int_0^s \lambda (X_G (\F_G^{\sigma}(y_1,p_1)))\d\sigma\\
-tH(y_1,p_1)-sG(y_1,p_1).
\end{multline*}
Since $G$ and $H$ commute, they are respectively first integral of the other which proves that  $G(y_0,p_0)=G(y_1,p_1)$ and that $H(y_0,p_0)=H(y_1,p_1)$.
Now let us consider the function $\psi$ defined from $ R=[0,s]\times [0,t] \subset \R{2}$ to $T^*M$ by
 $$\psi(\sigma,\sigma')=\F_G^{\sigma-s}\circ \F_H^{\sigma'}(y_0,p_0).$$
 Using Stokes' formula, the following holds:
 \begin{multline}\label{sto}
 \int_0^s \lambda (X_G (\F_G^{\sigma-s}(y_0,p_0)))\d\sigma+\int_0^t \lambda (X_H (\F_H^{\sigma}(y_0,p_0)))\d\sigma\\
 -\int_0^t \lambda (X_H (\F_H^{\sigma-t}(y_1,p_1)))\d\sigma-\int_0^s \lambda (X_G (\F_G^{\sigma}(y_1,p_1)))\d\sigma\\
 =\int_{\partial R}\psi^*\lambda
 =\int_R \psi^*\d\lambda
 =-\int_R \psi^*\Omega =0.
 \end{multline}
As a matter of fact, $\Omega$ vanishes identically on the tangent space to $\psi(R)$ which is at each point spanned by $X_G$ and $X_H$.\\
To put it all in a nutshell, we have proved that
\begin{multline*}
A=\int_0^tL_H\left(\gamma_4(\sigma),\dot{\gamma}_4(\sigma)\right)\d\sigma+\int_0^sL_G\left(\gamma_3(\sigma),\dot{\gamma}_3(\sigma)\right)\d\sigma\\
=\int_0^s \lambda (X_G (\F_G^{\sigma-s}(y_0,p_0)))\d\sigma+\int_0^t \lambda (X_H (\F_H^{\sigma}(y_0,p_0)))\d\sigma\\
 -tH(y_1,p_1)-sG(y_1,p_1)\\
=\int_0^s L_G(\gamma_1(\sigma),\dot{\gamma}_1(\sigma))\d\sigma+\int_0^t L_H(\gamma_2(\sigma),\dot{\gamma}_2(\sigma))\d\sigma\\
=A_G^s(x_0,y_0)+A_H^t(y_0,z_0)=\inf_{y\in M} A_G^s(x_0,y)+A_H^t(y,z_0).
\end{multline*}
Now, by the definition of the action functionals, the following inequality clearly holds:
$$\inf_{y\in M} A_G^s(x_0,y_0)+A_H^t(y_0,z_0)\geqslant \inf_{y\in M} A_H^t(x_0,y_0)+A_G^s(y_0,z_0).$$
By a symmetrical argument, the previous inequality is in fact an equality, which proves the lemma since $(x_0,z_0)$ was taken arbitrarily.
\end{proof}
The proof of the theorem is now straightforward:
\begin{proof}[proof of \ref{com}]
Let $u: M\to \mathbb{R}\cup \{-\infty,+\infty\}$ be any  function. By definition of the Lax-Oleinik semi-groups if $x\in M$ is a point, the following equalities hold:
\begin{multline*}
T_G^{-s}T_H^{-t}u(x)=\inf_{z\in M} \inf_{y\in M} u(z)+A_H^{t}(z,y)+A_G^{s}(y,x)\\
=\inf_{z\in M} \inf_{y\in M} u(z)+A_G^{s}(z,y)+A_H^{t}(y,x)
=
T_H^{-t}T_G^{-s}u(x).
\end{multline*}
\end{proof}
%
%
%
\section{Subsolutions and weak KAM solutions}

We explain here how most of the theory of subsolutions and viscosity solutions of the Hamilton-Jacobi equation can be adapted to the setting of two commuting Hamiltonians. Our presentation is mainly adapted from \cite{FaMa}. 
\\
Until now, in order to prove the commutation of the Lax-Oleinik semi-groups in its full generality (\ref{com}) , we did not assume any regularity or growth condition on the functions $u$ on which the semi-groups act. As a counterpart, we had to consider functions taking values in $\mathbb{R} \cup \{-\infty, +\infty\}$. As a matter of fact, the image of any real valued function by the Lax-Oleinik semi-group of a Tonelli Hamiltonian may have infinite values. Let us stress the fact that from now on, we will only be dealing with globally bounded functions or sub-solutions which in particular are globally Lipschitz. It is known that starting with a globally bounded or a Lipschitz function, $u:M\to \R{}$, the families of functions $(T^{-s}_G u)_{s\geqslant 0}$ and $(T^{-t}_H u)_{t\geqslant 0}$ are real valued functions. Moreover, it can be proved that when $u$ is Lipschitz, they are families of equi-Lipschitz functions (see \cite{FaMa} Proposition 3.2). 
\\
 Let us recall that if $\alpha$ is a real number, we say that $u:M\to\mathbb{R}$ is an $(\alpha,G)$-subsolution if
$$\forall (x,y,t)\in M^2\times \mathbb{R}_+,\  u(y)-u(x)\leqslant A_G^t(x,y)+t\alpha.$$
We denote by $\HH_G(\alpha)$ the set of $(\alpha,G)$-subsolutions. Of course we can also define analogously $(\alpha,H)$-subsolutions and we will denote by $\HH_H(\alpha)$ the set of such functions.  Finally, if $(\alpha,\alpha')\in \R{2}$, we will denote by
 $$\HH(\alpha,\alpha')=\HH_G(\alpha)\cap \HH_H(\alpha').$$
 Since $G$ and $H$ are Tonelli, for $\alpha$ and $\alpha'$ big enough, constant functions are both $(\alpha,G)$-subsolutions and $(\alpha',H)$-subsolutions, hence the set $\HH(\alpha,\alpha')$ is not empty. As a matter of fact, it follows from the Tonelli hypothesis on the Hamiltonians that the associated Lagrangians are also uniformly superlinear (see \cite{FaMa} lemma 2.1.) hence bounded below. Therefore there is a constant $C$ such that
 $$\forall (x,v)\in TM,\ \min(L_G(x,v),L_H(x,v))\geqslant C.$$
 Hence, for any absolutely continuous curve $\gamma: [0,t]\to M$ the following inequality holds:
 $$\int_0^t L_G(\gamma(\sigma),\dot{\gamma}(\sigma))\d\sigma\geqslant tC$$
 which may be rewritten as follows  
 $$0\leqslant \int_0^t L_G(\gamma(\sigma),\dot{\gamma}(\sigma))\d\sigma- tC.$$
 This implies directly that constant functions on $M$ are $(-C,G)$-subsolutions
 and obviously the same holds for $L_H$. 
  \\
 If $\alpha\in \R{}$, following Fathi, we say a function $u:M\to \R{}$ is a (negative) $(\alpha,G)$-weak KAM solution if 
 $$\forall t\geqslant 0,\  u=T_G^{-t}u +t\alpha.$$
 We denote by $\s_G^-(\alpha)$ the set of $(\alpha,G)$-weak KAM solutions. Obviously, we define analogously the notion of $(\alpha,H)$-weak KAM solution and the set $\s_H^-(\alpha)$.
 Let us now state Fathi's weak KAM theorem (we state it for $G$) (see \cite{Fa} for a proof in the compact case and \cite{FaMa} for a proof in the non compact case).
 \begin{Th}[weak KAM]
 There is a constant $\alpha_G[0]$ such that $\HH_G(\alpha[0])$ is not empty and if $\alpha<\alpha_G[0]$ then $\HH_G(\alpha)$ is empty. Moreover, the set $\s_G^-(\alpha_G[0])$ is not empty, that is:
 $$\exists u_-:M\to \R{},\ \forall s\geqslant 0,\  u_-=T_G^{-s}u_-+s\alpha_G[0].$$
 \end{Th}
 \begin{rem}\label{comm}\rm
 Let us mention that $\alpha_G[0]$ is called Ma\~ n\' e's critical value, therefore we call an $(\alpha_G[0],G)$-subsolution a $G$-critical subsolution. Moreover we will set $\s_G^-(\alpha_G[0])=\s_G^-$. If $M$ is compact, this notation is very natural, for if a function $u$ is an $(\alpha,G)$-weak KAM solution for some $\alpha$ then $\alpha=\alpha_G[0]$. However, let us stress that if $M$ is not compact, then as soon as $\alpha\geqslant \alpha_G[0]$ then $\s_G^-(\alpha)$ is not empty. As a matter of fact, using the Ma\~n\'e potential, it is possible to construct weak KAM solutions using a method inspired by the construction of Busemann functions in Riemannian geometry  (see \cite{Fa} corollary 8.2.3).
 \end{rem}
 Let us begin with an easy lemma:
 \begin{lm}\label{sta}
 The following assertions are true:
 \begin{enumerate}
 \item Let $u:M\to \R{}$ be a real valued function and $\alpha\in \R{}$, then $u\in  \HH_G(\alpha)$ if and only if
 $$\forall s\geqslant 0,\  u\leqslant T_G^{-s} u+s\alpha.$$
 \item For any $t\geqslant 0$ and $\alpha \in \R{}$, the set $\HH_G(\alpha)$ is stable by $T_H^{-t}$.
 \item For any $t\geqslant 0$ and $\alpha \in \R{}$, the set $\s^-_G(\alpha)$ is stable by $T_H^{-t}$.
 \end{enumerate}
 \end{lm}
 \begin{proof}
 The first part holds because by definition, $u$ is an $(\alpha,G)$-subsolution if and only if
 $$\forall (x,y,s)\in M^2\times \R{}_+,\  u(x)-u(y)\leqslant A_G^s(y,x)+s\alpha,$$
 which by taking an infimum on $y$ is equivalent to
 $$\forall (x,s)\in M\times \R{}_+,\  u(x)\leqslant \inf_{y\in M}u(y)+A_G^s(y,x)+s\alpha=T_G^{-s}u(x)+s\alpha.$$
 For the second part, by monotonicity of the Lax-Oleinik semi-group, using \ref{com}, we obtain that if $u\leqslant T_G^{-s} u+s\alpha$ then 
 $$T_H^{-t}u\leqslant T_H^{-t}(T_G^{-s} u+s\alpha)=T_H^{-t}(T_G^{-s} u)+s\alpha=T_G^{-s}(T_H^{-t} u)+s\alpha.$$
 The last point is a straightforward consequence of the commutation property of the semi-groups (\ref{com}). If for any positive $s$,  
$u= T_G^{-s} u+s\alpha$ then 
 $$T_H^{-t}u= T_H^{-t}(T_G^{-s} u+s\alpha)=T_H^{-t}(T_G^{-s} u)+s\alpha=T_G^{-s}(T_H^{-t} u)+s\alpha.$$

 \end{proof}

We now prove a version of the weak KAM theorem for commuting Hamiltonians when the manifold $M$ is compact. The proof is very similar to the proof of the classical weak KAM theorem (see \cite{Fa} or \cite{FaMa}). Let us recall that in the compact case, the Lax-Oleinik semi-groups are non-expansive for the infinity norm (\cite{Fa} proposition 4.6.5). This is in fact important since it will enable us to apply the following theorem of DeMarr (\cite{demarr}):
\begin{Th}[DeMarr]
Let $B$ be a Banach space and $(f_a)_{a\in A}$ a  family of commuting non-expansive continuous functions on $B$ which preserve a compact convex subset $C\subset B$, then these semi-groups have a common fixed point in $C$.
\end{Th}
\begin{Th}[double weak KAM]\label{dou}
 Let us assume $M$ is compact. There is a function $u_- :M\mapsto \R{}$ which is both a weak KAM solution for $G$ and $H$.
\end{Th}
 \begin{proof}
 Take $(\alpha,\alpha')\in \R{2}$ such that $\HH(\alpha,\alpha')$ is not empty. It is known (\cite{FaMa}) that $\HH(\alpha,\alpha')$ is made of  equi-Lipschitz functions. Therefore, by the Arzela-Ascoli theorem, the set 
 $$\HHH(\alpha,\alpha')=\HH(\alpha,\alpha')/\R{}\mathbbm{1}$$
  is compact for the compact open topology (where $\mathbbm{1}$ denotes the function constantly equal to $1$ on $M$).  Moreover, since 
$ \HH(\alpha,\alpha')$ is convex, the same holds for $\HHH(\alpha,\alpha')$. Finally, since by \ref{sta}, $ \HH(\alpha,\alpha')$ is stable by  the semi-groups, which  commute with the addition of constants, they induce two semi-groups which still commute and leave  $ \HHH(\alpha,\alpha')$ stable. Since the Lax-Oleinik semi-groups are non-expansive, we can therefore apply  DeMarr's theorem for commutative families of non-expansive maps (\cite{demarr}):
\begin{eqnarray*}
\exists u_-\in \HH(\alpha,\alpha'),\ \exists (\beta,\beta')\in \R{2}&,&\forall s\geqslant 0,\  u_-=T_G^{-s}u_-+s\beta\\
&&\forall t\geqslant 0,\  u_-=T_H^{-t}u_-+t\beta'.
\end{eqnarray*}
The function $u_-$ is the double weak KAM solution we are looking for.
\end{proof}
\begin{rem}\label{ext}\rm
Since $M$ is compact, using \ref{comm}, we obtain that in the previous proof, $\beta = \alpha_G[0]$ and $\beta'=\alpha_H[0]$.
\end{rem}
Actually, in the compact case the link between weak KAM solutions for $G$ and $H$ is much more simple to understand due to the following theorem which is a reformulation of \ref{inva}:
\begin{Th}\label{inv}
If $M$ is compact and $u_-:M\to \R{}$ is a weak KAM solution for $G$ then it is  also a weak KAM solution for $H$: $u_-\in \s^-$. In short, the following equalities hold:
$$\s^-_G=\s^-_H=\s^-.$$
\end{Th}
In order to prove this theorem we need to recall a few facts about Aubry-Mather theory. We call Mather set for $G$
$$\MM_G=\overline{\bigcup_{\mu}\supp \mu}$$
the closed union of all supports of minimizing probability measures on $T^*M$ invariant by $\F_G$ and $\M_G$ its projection on $M$. Clearly, $\MM_G$ is invariant by $\F_G$ but it actually is a symplectic invariant (see \cite{BeSy} and \cite{So}) and therefore it is also invariant by $\F_H$. Mather proved  (\cite{MaZ})  that $\MM_G$ is a compact Lipschitz graph over $\M_G$.  Finally, Fathi (\cite{Fa}) proved that if $u:M\to \R{}$ is a critical subsolution for $G$ and if $x\in \M_G$ is in the projected Mather set then the function $u$ is differentiable at $x$, and $(x,d_x u)\in \MM_G$, therefore the differential is independent of the  critical subsolution.\\
Finally, let us state that $\M_G$ is a uniqueness set for the stationary critical Hamilton-Jacobi equation associated with $G$, which means that if two $G$-weak KAM solutions coincide on $\M_G$, they are in fact equal.
\\
With these facts in mind, we are now able to prove the theorem:
\begin{proof}[proof of theorem \ref{inv}]
Let $u_0$ be a double weak KAM solution given by \ref{dou}. By what was mentioned above, for any $s\geqslant 0$ the function
$$v_s=T^{-s}_Hu_--T^{-s}_Hu_0=T^{-s}_Hu_-+s\alpha_H[0]-u_0$$
 is differentiable on $\M_G$ with a vanishing differential. Let $(x,p)\in \MM_G$ and set 
 $$\forall s\in \R{},\  (x(s),p(s))=\F_H^s(x,p)\in \MM_G,$$
  then it is known (see  \cite{Fa} (4.11.1), \cite{fatfig07}  corollary B.20, \cite{Be2}) that
 \begin{multline*}
 \forall s>0,\  v_s(x)=T^{-s}_Hu_-(x)-T^{-s}_Hu_0(x)\\
 =u_-(x(-s))+\int_{-s}^0 L_H (\L_H^{-1}(\F_H^\sigma(x,p)))\d\sigma \\
 -u_0(x(-s))-\int_{-s}^0 L_H (\L_H^{-1}(\F_H^\sigma(x,p)))\d\sigma\\
 =u_-(x(-s))-u_0(x(-s))=v_0(x(-s)).
 \end{multline*}
 Since the trajectory $s\mapsto x(s), s\in \mathbb{R}$ is $C^2$, has its image included in $\M_G$ and the function $v_0$ has a vanishing differential on it, we can deduce that $v_0$ is constant on the image of   $s\mapsto x(s)$. Therefore, 
 \begin{eqnarray}
 \forall s>0,\  v_s(x)&=&T^{-s}_Hu_-(x)-T^{-s}_Hu_0(x)\nonumber\\
&=&T^{-s}_Hu_-(x)+s\alpha_H[0]-u_0(x)\nonumber\\
&=&u_-(x)-u_0(x).\nonumber
\end{eqnarray}
In short, 
\begin{equation}\label{==}
\forall x\in \M_G,\  u_-(x)=T^{-s}_Hu_-(x)+s\alpha_H[0].
\end{equation}
We have proved the invariance on $\M_G$, it remains to prove the same on its complementary. 
But the equality of $u_-$ and of  $T^{-s}_Hu_-+s\alpha_H[0]$ everywhere follows directly from the facts that they are both $G$-weak KAM solutions (\ref{sta}) and that two $G$-weak KAM solutions that coincide on $\M_G$ must coincide everywhere (\cite{Fa} Theorem 4.12.6).
\end{proof}
\begin{co}\label{coinv}
If $M$ is compact and if $u_-$ is a $G$-weak KAM solution then the graph of the differential of $u_-$, $\Gamma(u_-)$ verifies the following:
\begin{equation}\label{rompe}
\forall t\leqslant 0, \  \F_H^{-t}(\Gamma(u_-))\subset \Gamma(u_-).
\end{equation}
 Note that $(\F_H^{-t})_{t\geqslant 0}$ is a one-parameter semi-group of symplectomorphisms preserving $G$. Moreover, the same holds for $\overline{\Gamma(u_-)}$.
\end{co}
\begin{proof}
It is proved in \cite{Fa} (Theorem 4.13.2) that if $u_-\in \s^-_H$ is a weak KAM solution for $H$, then
$$\forall t\leqslant 0, \  \F_H^{-t}(\Gamma(u_-))\subset \Gamma(u_-).$$
Let us recall the main steps of this proof. If $x\in M$ is a differentiability point of $u_-$, the following holds:
$$ \forall s>0,\  u_-(x)=T^{-s}_Hu_-(x)
 =u_-(x(-s))+\int_{-s}^0 L_H (\L_H^{-1}(\F_H^\sigma(x,\d_x u_-)))\d\sigma,$$
 where we have used the following notation: 
 $$\forall s\in \R{},\  (x(s),p(s))=\F_H^s(x,\d_x u).$$
 Moreover, it can be proved that for all $s\geqslant 0$, the point $x(-s)$ is then a differentiability point $u_-$ which verifies $\d_{x(-s)}u_-=p(-s)$ (one could use the fact that $u_-$ is locally semi-concave and that $p(-s)$ is in the sub-differential of $u_-$ at $x(-s)$ (\cite{Fa} proposition 4.11.1)). This proves the inclusion (\ref{rompe}) since by \ref{inv} any weak KAM solution for $G$ is also a weak KAM solution for $H$. The end of the corollary is straightforward.
 \end{proof}

\begin{rem}\rm
The proof of theorem \ref{inv} is very similar to the one given in $\cite{Mad}$ of a result (in the compact case) concerning the stability of weak KAM solutions by diffeomorphisms of the base space. More precisely, let $\Gamma^G(M)$ denote the set of $C^1$ diffeomorphisms which preserve $G$ equipped with the topology of uniform convergence on compact subsets. Let $\Gamma^G_0$ denote the identity component of $\Gamma^G$. Then if $g\in \Gamma^G_0$, any weak KAM solution for $G$ is stable by $g$. In this case, we have corollary \ref{coinv} as a  similar statement. Indeed it asserts that for a certain class of symplectomorphisms preserving $G$, the graph of the differential of a weak KAM solution is also, in a weak sense, stable by these symplectomorphisms.
\end{rem}


\section{Positive time Lax-Oleinik semi-groups and $C^{1,1}$ subsolutions}
As usual in weak KAM theory, there is a positive time analog for every result proved. Let us see how this applies to commuting Hamiltonians. Here again, we follow the exposition from \cite{Fa}. \\
Given a Tonelli Hamiltonian, $G:T^*M\to \R{}$ and his associated Lagrangian $L_G:TM\to \R{}$, we can define the symmetrical Hamiltonian and Lagrangian defined as follows:
$$\forall (x,p)\in T^*M,\  \widehat{G}(x,p)=H(x,-p),$$
$$\forall (x,v)\in TM,\ \widehat{L}_G(x,v)=L(x,-v).$$
Obviously, $\widehat{G}$ and $\widehat{L}_G$ are once again Legendre transform of one another and are still Tonelli. We may now define the positive time Lax-Oleinik semi-group of a function $u:M\to \R{}$:
$$\forall s>0,\ \forall x\in M,\  T_G^{+s}u(x)=-T_{\widehat{G}}^{-s}(-u)(x).$$
Now, if $H$ is another Tonelli Hamiltonian which Poisson commutes with $G$, it is clear that $\widehat{G}$ and $\widehat{H}$ also Poisson commute. Therefore, we have the following results:
\begin{Th}[G. Barles-A. Tourin]\label{com+}
 The positive time Lax-Oleinik semi-groups commute, that is, if $u:M\to \R{}$ is a function and $s,t$ are two positive real numbers then 
$$T_G^{+s}T_H^{+t}u=T_H^{+t}T_G^{+s}u.$$
\end{Th}
If $\alpha\in \R{}$, following Fathi, we say a function $u:M\to \R{}$ is a positive time $(\alpha,G)$-weak KAM solution if 
 $$\forall t\geqslant 0,\  u=T_G^{+t}u -t\alpha.$$
 We denote by $\s_G^+(\alpha)$ the set of positive time $(\alpha,G)$-weak KAM solutions. Obviously, we define analogously the notion of positive time $(\alpha,H)$-weak KAM solution and the set $\s_H^+(\alpha)$.
 Let us now state Fathi's weak KAM theorem (we state it for $G$) 

\begin{Th}[positive time weak KAM]
 The set $\s_G^+(\alpha_G[0])$ is not empty, that is:
 $$\exists u_+:M\to \R{},\ \forall s\geqslant 0,\  u_+=T_G^{+s}u_+-s\alpha_G[0].$$
 \end{Th}
 
 \begin{rem}\label{comm+}\rm
 Let us mention that $\alpha_G[0]$ is again  Ma\~ n\' e's critical value. Moreover we will set $\s_G^+(\alpha_G[0])=\s_G^+$. If $M$ is compact, this notation is very natural, since as for negative time, if a function $u$ is a positive time $(\alpha,G)$-weak KAM solution for some $\alpha$ then $\alpha=\alpha_G[0]$. 
 \end{rem}
 \begin{lm}\label{sta+}
 The following assertions are true:
 \begin{enumerate}
 \item Let $u:M\to \R{}$ be a real valued function and $\alpha\in \R{}$, then $u\in  \HH_G(\alpha)$ if and only if
 $$\forall s\geqslant 0,\  u\geqslant T_G^{+s} u-s\alpha.$$
 \item For any $t\geqslant 0$ and $\alpha \in \R{}$, the set $\HH_G(\alpha)$ is stable by $T_H^{+t}$.
 \item For any $t\geqslant 0$ and $\alpha \in \R{}$, the set $\s^+_G(\alpha)$ is stable by $T_H^{+t}$.
 \end{enumerate}
 \end{lm}

\begin{Th}\label{inv+}
If $M$ is compact and $u_+:M\to \R{}$ is a positive time weak KAM solution for $G$ then it is  also a  positive time weak KAM solution for $H$: $u\in \s^+$. In short, the following equalities hold:
$$\s^+_G=\s^+_H=\s^+.$$
\end{Th}
Equipped with the positive time Lax-Oleinik semi-groups, we are now able to prove existence theorems of $C^{1,1}$ common subsolutions for $G$ and $H$ (this is theorem \ref{subs}).
\begin{Th}[existence of common $C^{1,1}$ subsolutions]
Assume that the pair  $(\alpha,\alpha')\subset \R{2}$ is such that $\HH(\alpha,\alpha')\neq \varnothing$, then there is a locally $C^{1,1}$ function in $\HH(\alpha,\alpha')$. Moreover, $\HH(\alpha,\alpha')\cap C^{1,1}(M,\R{})$ is dense in $\HH(\alpha,\alpha')$ for the compact open topology.
\end{Th}
\begin{proof}
The proof is just a simple adaptation of \cite{fatfigrif} which itself is very much inspired from \cite{Be1}. The idea is to use successively positive and negative Lax-Oleinik semi-groups in order to realize a kind of Lasry-Lions regularization. More precisely, if $u_0\in \HH(\alpha,\alpha')$, it is proved in \cite{fatfigrif} that for a suitable choice of "small" positive constants, $(\e^-_k)_{k\in \mathbb{N}^*}$ and $(\e^+_k)_{k\in \mathbb{N}^*}$,the sequence of functions 
$$\forall n\in \mathbb{N},\  u_n=T_G^{+\e_n}T_G^{-\e_n}\cdots T_G^{+\e_1}T_G^{-\e_1}u_0$$
converges (for the compact open topology) to a $C^{1,1}$ function $u_\infty$ which is an $(\alpha,G)$-subsolution. \\
Moreover, by \ref{sta} and \ref{sta+}, the functions $u_n$ are also $(\alpha',H)$-subsolutions, which proves that $u_\infty$ is itself an $(\alpha',H)$-subsolutions ($\HH(\alpha,\alpha')$ is closed for the compact open topology). Therefore, $u_\infty$ is a $C^{1,1}$ function which belongs to $\HH(\alpha,\alpha')$.\\
To prove the density result, just notice that by continuity of the Lax-Oleinik semi-groups as maps from $[0,+\infty[\times \HH(\alpha,\alpha')$ to  $\HH(\alpha,\alpha')$ (for the compact open topology on $\HH(\alpha,\alpha')$, see \cite{FaMa} Proposition 3.3) , by taking smaller sequences   $(\e^-_k)_{k\in \mathbb{N}^*}$ and $(\e^+_k)_{k\in \mathbb{N}^*}$, we can actually construct $u_\infty$ arbitrarily close to $u_0$.
\end{proof}

\section{More on the compact case}\label{4} 
Throughout this section, we assume $M$ is compact. Moreover, up to adding constants to $G$ and $H$, we will assume that $\alpha_G[0]=\alpha_H[0]=0$. In \cite{Fa3} Fathi proved the following:
\begin{Th}[paired weak KAM solutions]\label{paired}
Given a critical subsolution $u_G$ for $G$ (resp. $u_H$ for $H$), there exist a unique  negative weak KAM solution $u^-_G$ and a unique positive weak KAM solution $u^+_G$ (resp. $u^-_H$ and $u^+_H$) such that
$$u_{G|\mathcal{M}_G}=u^-_{G|\mathcal{M}_G}=u^+_{G|\mathcal{M}_G},$$
(resp. $u_{H|\mathcal{M}_H}=u^-_{H|\mathcal{M}_H}=u^+_{H|\mathcal{M}_H}$).
We denote this relation by $u_G^-\sim_Gu_G^+$ (resp. $u_H^-\sim_Hu_H^+$). Moreover, we have that 
$$\lim_{t\to +\infty}T^{-t}_Gu_G=u_G^-,$$
$$\lim_{t\to +\infty}T^{+t}_Gu_G=u_G^+,$$
(resp. $\lim_{t\to +\infty}T^{-t}_Hu_H=u^-_H$ and $\lim_{t\to +\infty}T^{+t}_Hu_H=u^+_H$) where the limits hold for the infinity norm.
\\
Finally, let us mention that paired weak KAM solutions are characterized by those limits, more precisely 
$$u_G^-\sim_Gu_G^+\Longleftrightarrow \lim_{t\to +\infty}T^{-t}_Gu^+_G=u_G^-\Longleftrightarrow \lim_{t\to +\infty}T^{+t}_Gu^-_G=u_G^+,$$
(resp. $u_H^-\sim_Hu_H^+\Longleftrightarrow \lim_{t\to +\infty}T^{-t}_Hu^+_H=u_H^-\Longleftrightarrow \lim_{t\to +\infty}T^{+t}_Hu^-_H=u_H^+$).
\end{Th}
Let us define (for $G$) the Aubry and Ma\~n\' e sets and the Peierls barrier:
 \begin{df}\rm
 The Aubry set is defined by
 $$\widehat {\mathcal{A}}_G=\bigcap_{u^-\sim_G u^+}\left\{(x,d_x u^-), u^-(x)=u^+(x) \right\},$$
 the Ma\~n\' e set is defined by 
  $$\widehat {\mathcal{N}}_G=\bigcup_{u^-\sim_G u^+}\left\{(x,d_x u^-), u^-(x)=u^+(x) \right\},$$
  and finally the Peierls barrier is defined by
  $$\forall (x,y)\in M^2, \ h(x,y)=\sup_{u^-\sim_G u^+} u^-(y)-u^+(x).$$
 \end{df}
 \begin{rem}\rm
 In the definitions of the Aubry and Ma\~n\' e sets it must be justified why the differentials of $u^-$ exist. This comes from the facts that if $(u^-,u^+)$ are paired weak KAM solutions, then $u^-$ (resp. $u^+$) is locally semi-concave (resp. locally semi-convex, see \cite{Fa} proposition 6.2.1), with $u^-\geqslant u^+$. Hence $u^-$ is differentiable on the set $\{x\in M,u^-(x)=u^+(x)\}$.
 \\
  Usually, the Aubry, Ma\~n\' e sets  and the Peierls barrier are rather defined on the tangent bundle of $M$ using the Lagrangian setting and the action functionals. However, for simplicity of the exposition, we only give here these equivalent definitions from the weak KAM point of view.
 \end{rem}
  
The result we are going to prove is that both relations $\sim_G$ and $\sim_H$ are the same:
\begin{Th}\label{same}
Let $u^-$ and $u^+$ be a negative and a positive weak KAM solution, then if $u^-\sim_G u^+$ then $u^-\sim_H u^+$.
\end{Th}
Before giving the proof of this theorem, let us recall another result of Fathi (\cite{Fa1}). Note that we are still assuming that $\alpha_G [0]=\alpha_H[0]=0$.
\begin{Th}\label{conv}
Let $u:M\to \R{}$ be a continuous function, then the functions $T^{-t}_Gu$ converge as $t$ goes to $+\infty$ to a function $u^-$ which moreover is a negative weak KAM solution. Obviously, the same holds for the time positive Lax-Oleinik semi-group and for the semi-groups associated with
 $H$.
 \end{Th}
 Now, using  the last part of \ref{paired}, the proof of \ref{same} is a direct consequence of the following proposition (and of its analog for the positive time Lax-Oleinik semi-groups):
 \begin{pr}
 Let $u:M\to \R{}$ be a continuous function, and let $u_G^-$ be the limit of the $T^{-t}_Gu$ (resp. $u_H^-$ be the limit of the $T^{-t}_Hu$). Then we have that $u^-_G=u^-_H$.
 \end{pr}
 \begin{proof}
 We begin with the proof that for any positive $s$, the functions $T^{-t}_G (T^{-s}_Hu)$ still converge to $u^-_G$ as $t$ goes to infinity.
 \\
 It is a simple consequence of the commutation property of the semi-groups (\ref{comu}). As a matter of fact, for all $t$ the following holds
 $$T^{-t}_G (T^{-s}_Hu)=T^{-s}_H (T^{-t}_Gu)$$
 and by continuity of the Lax-Oleinik semi-group, the functions $T^{-s}_H (T^{-t}_Gu)$ converge to 
 $T^{-s}_H (u^-_G)=u^-_G$ by \ref{inv}.
 \\
 Now, let us recall that the Lax-Oleinik semi-groups are $1$-Lipschitz for the infinity norm, therefore we have that
 $$\forall (s,t)\in \mathbb{R}_+^2,\  \|T^{-t}_G (T^{-s}_Hu)-u^-_H\|_\infty=\|T^{-t}_G (T^{-s}_Hu)-T^{-t}_G u^-_H\|_\infty\leqslant \| (T^{-s}_Hu)- u^-_H\|_\infty.$$
 Letting $t$ go to $+\infty$ we obtain that
 $$\forall s\in \mathbb{R}_+,\ \| u^-_G- u^-_H\|_\infty\leqslant \| (T^{-s}_Hu)- u^-_H\|_\infty\xrightarrow[s\to +\infty]{} 0,$$
 this proves the result.
 \end{proof}
  We end this section by the following theorem which is a straight consequence of the definitions and of \ref{same}:
  \begin{Th}\label{equal} 
  The following equalities hold:
  $$\widehat {\mathcal{A}}_G=\widehat {\mathcal{A}}_H,\ 
  \widehat {\mathcal{N}}_G=\widehat {\mathcal{N}}_H,\ 
  h_G=h_H.$$
  \end{Th}
  
  \section{Flats of Mather's $\alpha$ function}
  In this section, the underlying manifold $M$ is still compact. We will need the following notation: if $H$ is a Tonelli Hamiltonian, we will denote the set $\AA_H\subset TM$ defined as follows:
  $$\AA_H=\L_H^{-1}\left(\AAA_H\right).$$
  Given a Tonelli Hamiltonian $H$ and a closed $1$-form $\om : T^*M\to \R{}$, Mather noticed in \cite{MaZ} that the Hamiltonian $H_\om$ defined by 
  $$\forall (x,p)\in T^* M, \  H_\om (x,p)=H(x,p+\om_x)$$
is still Tonelli, therefore it admits a critical value and this critical value depends only on the cohomology class of $\om$ which we denote $[\om]\in H^1(M,\R{})$. We call $\alpha_H[\om]$ the critical value of $H_\om$. Mather also proves that the function $\alpha_H :   H^1(M,\R{})\to \R{}$ is convex and superlinear. The function $\alpha_H$ is sometimes called effective Hamiltonian in homogenization theory.  A flat of the function $\alpha_H$ is a convex set $C\subset  H^1(M,\R{})$ on which $\alpha_H$ is linear. The following theorem is proved in \cite{Mas} and \cite{Be3}:

\begin{Th}\label{aubry}
Assume that Mather's $\alpha_H$ function is affine between two cohomology classes $[\om_0]$ and $[\om_1]$, that is
$$\forall t\in (0,1), \ \alpha_H [t\om_0 +(1-t)\om_1]=t\alpha_H[\om_0]+(1-t)\alpha_H[\om_1]$$
then the following holds:
$$\forall t\in (0,1),\  \AA_{H_{\om_t}}\subset \AA_{H_{\om_0}}\cap \AA_{H_{\om_1}}$$
where we have used the notation $\om_t=t\om_0 +(1-t)\om_1$.\\
In particular, the Aubry sets are constant in the relative interior of a flat of the $\alpha_H$ function and the $\alpha_H$ function must be constant on this flat.
\end{Th}

We now want to study the relation between flats of $\alpha$ functions of two commuting Hamiltonians. The following proposition shows that the question is legitimate :

  \begin{pr}\label{como}
Let $G$ and $H$ be two commuting  Tonelli Hamiltonians. Then for any closed one form, $\om$ on $M$, we have that $G_\om$ and  $ H_\om$ Poisson commute.
\end{pr}
\begin{proof}
It is a direct consequence of the definition of the Poisson bracket and the fact that when $\om$ is closed, the map $\psi_\om : T^*M\to T^*M$ defined by
$$\forall (x,p)\in T^*M, \ \psi_\om (x,p)=(x,p+\om_x)$$
is symplectic.
\end{proof}
  
  We may now state the main result of this section:
\begin{Th}
Let $M$ be a compact $C^2$ closed manifold and $G,H$ two commuting Tonelli Hamiltonians on $T^*M$. Let us denote by $C_G\subset H^1(M,\R{})$ a flat of $\alpha_G$ on which it  is therefore constant. Then $C_G$ is also a flat of $\alpha_H$.
\end{Th}
\begin{proof}
Let us consider $\om_1,\om_2$ two closed forms whose cohomology classes belong to the relative interior of $C_G$. As seen in \ref{aubry}, we then have the following equality:
$$\AA_{G_{\om_1}}= \AA_{G_{\om_2}}$$
which after taking the Legendre transform yields
\begin{equation}\label{commo}
\AAA_{G_{\om_1}}+\om_1= \AAA_{G_{\om_2}}+\om_2
\end{equation}
 (where we denote $\AAA_{G_{\om_1}}+\om_1=\{(x,p+\om_{1,x}), (x,p)\in \AAA_{G_{\om_1}}\}$ and $\AAA_{G_{\om_2}}+\om_2=\{(x,p+\om_{2,x}), (x,p)\in \AAA_{G_{\om_2}}\}$). Now using \ref{commo}, \ref{como} and \ref{equal} we obtain that
$$\AAA_{H_{\om_1}}+\om_1=\AAA_{G_{\om_1}}+\om_1= \AAA_{G_{\om_2}}+\om_2= \AAA_{H_{\om_2}}+\om_2.$$
But we also know that if $(x,p)\in\AAA_{H_{\om_1}}$ then $(x,p')=(x,p+(\om_1-\om_2)(x))\in \AAA_{H_{\om_2}}$ and the following holds (where we use Carneiro's theorem (\cite{Car}) stating that the Aubry set lies in the critical energy level of the Hamiltonian):
\begin{multline*}
\alpha_H[\om_1]=H_{\om_1}(x,p)\\
=H(x,p+\om_{1,x} )=H(x,p'+\om_{2,x} )\\
=H_{\om_2}(x,p')=\alpha_H[\om_2].
\end{multline*}
This proves that $\alpha_H$ is constant on $C_G$.
\end{proof}

 \bibliography{commuting}
\bibliographystyle{alpha}

\end{document}